\documentclass[12pt]{amsart}
\usepackage{amsmath}
\usepackage{amscd}
\usepackage{amssymb}
\usepackage{amsfonts}

\newtheorem{theorem}{Theorem}[section]
\newtheorem{lemma}[theorem]{Lemma}

\newtheorem{proposition}[theorem]{Proposition}

\theoremstyle{definition}

\theoremstyle{remark}
\newtheorem{remark}[theorem]{Remark}
\numberwithin{equation}{section}

\begin{document}
\title[Sharp Gaussian upper bounds for Schr\"odinger heat kernel]
{Sharp Gaussian upper bounds for Schr\"odinger heat kernel on gradient shrinking Ricci solitons}
\author{Jia-Yong Wu}
\address{Department of Mathematics, Shanghai University, Shanghai 200444, China}
\email{wujiayong@shu.edu.cn}
\thanks{}
\dedicatory{}
\date{\today}

\begin{abstract}
On gradient shrinking Ricci solitons, we observe that the study of
Schr\"odinger heat kernel seems to be more natural than the classical
heat kernel. In this paper we derive sharp Gaussian upper bounds for
the Schr\"odinger heat kernel on complete gradient shrinking Ricci
solitons. As applications, we prove sharp upper bounds for the Green's
function of the Schr\"odinger operator. We also prove sharp lower bounds
for eigenvalues of the Schr\"odinger operator. These sharp cases are all
achieved at Euclidean Gaussian shrinking Ricci solitons.
\end{abstract}
\maketitle

\section{Introduction}
In this paper we will investigate Gaussian upper estimates for Schr\"odinger heat kernels
on complete gradient shrinking Ricci solitons and their applications. Let $(M,g)$ be an
$n$-dimensional complete Riemannian manifold and let $f$ be a smooth function on $(M,g)$.
The triple $(M, g, f)$ is called a complete \emph{gradient shrinking Ricci soliton}
(see \cite{[Ham]}) if
\begin{align}\label{Eq1}
\mathrm{Ric}+\mathrm{Hess}\,f=\tfrac 12g,
\end{align}
where $\text{Ric}$ is the Ricci curvature of $(M,g)$ and $\text{Hess}\,f$ is the Hessian of
$f$. The function $f$ is often called a potential for the gradient Ricci soliton. Gradient Ricci
solitons play an important role in the Ricci flow theory \cite{[Ham],[Cao]} and Perelman's
\cite{[Pe],[Pe2],[Pe3]} resolution of the Poincar\'e conjecture and the geometrization
conjecture.

In this paper, we will consider Schr\"odinger heat kernels of the operator
\[
L:=-\Delta+a\mathrm{R}
\]
on a gradient shrinking Ricci soliton, where $\Delta$, $\mathrm{R}$ and $a$
denote the Laplace operator, the scalar curvature of $(M,g)$ and a positive
constant, respectively. Let $h^R(x,y,t): M\times M\times\mathbb{R}^+\to \mathbb{R}$
be the \emph{Schr\"odinger fundamental solution} for the operator $L$. That is, for each $y\in M$,
$h^{\mathrm{R}}(x,y,t)=u(x,t)$ is a smooth solution to the Schr\"odinger-type heat equation
\[
(\partial_t+L) u=0
\]
with the initial condition $\lim_{t\to0}u(x,t)=\delta_y(x)$, where $\delta_y(x)$
is the delta function defined by
\[
\int_M\phi(x)\delta_y(x)dv=\phi(y)
\]
for any $\phi\in C_0^{\infty}(M)$. We say that a Schr\"odinger fundamental solution
$H^{\mathrm{R}}(x,y,t)$ for the operator $L$ is the \emph{Schr\"odinger heat kernel}
(also called the \emph{minimal positive fundamental solution}) if $H^{\mathrm{R}}(x,y,t)$
is positive and if for every positive fundamental solution $h^R(x,y,t)$ we have
$h^R(x,y,t)\ge H^R(x,y,t)$. Throughout this paper, we denote by $H^{\mathrm{R}}(x,y,t)$
the Schr\"odinger heat kernel for the operator $L$. We will see that, when the scalar
curvature is bounded from above by a constant, the Schr\"odinger heat kernel
$H^{\mathrm{R}}(x,y,t)$ always exists on $(M,g,f)$ (for the explanation,
see Section \ref{sec2}).

Several factors motivated us to study the Schr\"odinger operator $L$
instead of the classical Laplace operator $-\Delta$ on $(M,g,f)$. First, Perelman's
geometric operator $-\Delta+\frac 14\mathrm{R}$ \cite{[Pe]} is a Schr\"odinger operator
and was widely considered in the Ricci flow theory. Given that shrinking Ricci solitons
are the self-similar solutions to the Ricci flow \cite{[Ham]}, the Schr\"odinger
operator $L$ seems to be more natural compared with the Laplace operator for gradient
Ricci solitons. Second, the gradient shrinking Ricci soliton is related to the Yamabe
invariant \cite{[AIL],[Pe]} associated to the conformal Laplacian
$-\Delta+\frac{n-2}{4(n-1)}\mathrm{R}$, which is a special case of Schr\"odinger
operator. Third, for gradient shrinking Ricci solitons, Li and Wang \cite{[LiWa]}
proved a Sobolev inequality including a scalar curvature term, which inspired us to
consider the Schr\"odinger operator instead of the Laplace operator.

On a gradient shrinking Ricci soliton $(M,g,f)$, according to a nice observation
of Carrillo and Ni (Theorem 1.1 in \cite{[CaNi]}), by adding a constant to $f$, without loss
of generality, we may assume (see the explanation in Section \ref{sec2} or \cite{[LiWa]})
\begin{equation}\label{Eq2}
\mathrm{R}+|\nabla f|^2=f \quad\text{and}\quad \int_M (4\pi)^{-\frac n2}e^{-f} dv=e^{\mu},
\end{equation}
where $\mu=\mu(g,1)$ is the entropy
functional of Perelman \cite{[Pe]}. For the Ricci flow, Perelman's entropy functional
is time-dependent, but on a fixed gradient shrinking Ricci soliton it is constant and finite.

In this paper we mainly prove a Gaussian upper bound for the Schr\"odinger heat
kernel $H^{\mathrm{R}}(x,y,t)$, similar to the classical Gaussian heat kernel
estimate for Laplace operator on manifolds. This result will be useful for
understanding the geometry and topology of gradient shrinking Ricci solitons.
\begin{theorem}\label{Mainup2}
Let $(M,g, f)$ be an $n$-dimensional complete gradient shrinking Ricci soliton satisfying
\eqref{Eq1} and \eqref{Eq2} with scalar curvature $\mathrm{R}$ bounded from above by a
constant. For any $c>4$, there exists a constant $A=A(n, c)$ depending on $n$ and $c$
such that the Schr\"odinger heat kernel of the operator $-\Delta+a\mathrm{R}$ with
$a\geq \frac 14$ satisfies
\begin{equation}\label{low2}
H^{\mathrm{R}}(x,y,t)\le\frac{A e^{-\mu}}{(4\pi t)^{\frac n2}}\exp\left(-\frac{d^2(x,y)}{ct}\right)
\end{equation}
for all $x,y\in M$ and $t>0$, where $\mu$ is Perelman's entropy functional.
\end{theorem}
\begin{remark}\label{remex}
The upper assumption on scalar curvature only guarantees that the Schr\"odinger heat
kernel possesses similar propositions of the classical heat kernel of the Laplace
operator, such as existence, semigroup property, eigenfunction expansion, etc.
(see Section \ref{sec2}). It seems not to be directly used in our proof of Theorem
\ref{Mainup2}. It is interesting to ask if the Schr\"odinger heat kernel still
exists on complete noncompact gradient shrinking Ricci solitons when the scalar
curvature assumption is removed.
\end{remark}
\begin{remark}\label{remex2}
The classical heat kernel of the Laplace operator on an $n$-dimensional Euclidean
space $\mathbb{R}^n$ is
\[
H(x,y,t)=\frac{1}{(4\pi t)^{\frac n2}}\exp\left(-\frac{|x-y|^2}{4t}\right)
\]
for all $x,y\in \mathbb{R}^n$ and $t>0$. This indicates the Gaussian upper estimate
\eqref{low2} is sharp with respect to $t$ (at this time $\mathrm{R}=\mu=0$).
\end{remark}

\begin{remark}\label{remex3}
Li and Wang \cite{[LiWa]} recently proved a non-Gaussian upper bound for the
heat kernel of the Ricci flows induced by gradient shrinking Ricci solitons.
Our result focuses on a fixed gradient shrinking Ricci soliton and contains
a Gaussian exponential factor.
\end{remark}

For the heat kernel of the Laplace operator, Cheng, Li and Yau \cite{[CLY]} ever
proved upper Gaussian estimates on manifolds satisfying bounded sectional curvature and a
lower bound of the injectivity radius, which was later generalized by Cheeger, Gromov
and Taylor \cite{[CGT]} to manifolds with the Ricci curvature bounded below. In 1986,
Li and Yau \cite{[LY]} used the gradient estimate technique to derive sharp
Gaussian upper and lower bounds on manifolds with nonnegative Ricci curvature.
In 1990s, Grigor'yan \cite{[Gr]} and Saloff-Coste \cite{[Sal]} independently proved
similar estimates on manifolds satisfying the volume doubling property and the
Poincar\'e inequality, by using the Moser iteration technique. Davies \cite{[Da2]}
further developed Gaussian upper bounds under a mean value property assumption.
Recently, the first author and P. Wu \cite{[WuW]} applied De Giorgi-Nash-Moser
theory to derive sharp Gaussian upper and lower estimates for the weighted
heat kernel on smooth metric measure spaces with nonnegative Bakry-\'Emery Ricci
curvature. For the heat kernel of a general Schr\"odinger operator $-\Delta+Q$
for some $Q\in C^\infty(M)$, many authors studied global bounds for the heat
kernel on manifolds. The interested readers are referred to \cite{[Car],[Da],[DS],[DS2],[Gr2],[LY],[Pi],[Sa],[Zh1],[Zh2],[Zh3],[Zh4]}
and references therein.

The proof strategy for Theorem \ref{Mainup2} seems to be different from the
above-mentioned methods, and here its proof mainly includes two steps. In the first step
we apply a local Logarithmic Sobolev inequality for shrinking Ricci solitons to
give an upper bound for the Schr\"odinger heat kernel (see Theorem \ref{Mainup}),
which is motivated by the argument valid for manifolds \cite{[Da]}. In the second step
we extend the upper bound for the Schr\"odinger heat kernel to its upper
bound with a Gaussian exponential factor, whose argument involves upper
estimates for a weighted integral of the Schr\"odinger heat kernel (see Proposition
\ref{pro1}), by using a delicate iteration technique due to Grigor'yan \cite{[Gr1]}.

Below we give two applications of Schr\"odinger heat kernel estimates. On one
hand we will derive upper bounds for the Green's function of the Schr\"odinger operator
on gradient shrinking Ricci solitons. Recall that for Riemannian mainfolds, Li and
Yau \cite{[LY]} applied the gradient estimate technique to prove two-sided bounds of
classical Green's functions. Grigor'yan \cite{[Gr2]} studied two-sided bounds of
abstract Green's functions when some doubling property holds. Recently many properties
of weighted Green's functions on weighted manifolds have been investigated; see for
example \cite{[Gr2]}, \cite{[Pig]}, \cite{[WuW]} and references therein. Similar to the
manifold case, on a complete gradient shrinking Ricci soliton $(M,g,f)$, the Green's
function of the Schr\"odinger operator $-\Delta+a\mathrm{R}$ with $a\geq \frac 14$ is
defined by
\[
G^{\mathrm{R}}(x,y):=\int^\infty_0H^{\mathrm{R}}(x,y,t)dt
\]
if the integral on the right-hand side converges. Hence,
\begin{theorem}\label{Green}
Let $(M,g,f)$ be an $n$-dimensional $(n\ge 3)$ complete gradient shrinking Ricci soliton
satisfying \eqref{Eq1} and \eqref{Eq2} with scalar curvature $\mathrm{R}$ bounded from
above by a constant. If $G^{\mathrm{R}}(x,y)$, $x,y\in M$, exists,
then for any $c>4$, there exists a constant $C(n,c)$ depending on
$n$ and $c$, such that
\begin{equation}\label{Greenest}
G^{\mathrm{R}}(x,y)\le C(n,c)e^{-\mu}d(x,y)^{2-n},
\end{equation}
where $d(x,y)$ is the distance function from $x$ to $y$ and $\mu$ is Perelman's entropy functional.
\end{theorem}
\begin{remark}
The above exponent $2-n$ is sharp. Indeed, on the Gaussian shrinking Ricci soliton
$(\mathbb{R}^n,g_E,\frac{|x|^2}{4})$, where $g_E$ is the standard Euclidean metric,
we have $\mathrm{R}=0$ and $\mu=0$. In this case $G^{\mathrm{R}}(x,y)$ is just the
Euclidean Green's function given by $G^{\mathrm{R}}(x,y)=C(n)d(x,y)^{2-n}$
for some positive constant $C(n)$, where $n\ge3$.
\end{remark}

On the other hand, we will apply the Schr\"odinger heat kernel estimate to prove lower
bounds for eigenvalues of the operator Schr\"odinger $L$ on compact gradient
shrinking Ricci solitons, by adapting the argument for the Laplace operator on
manifolds \cite{[LY]}. Some basic spectral properties of the Schr\"odinger operator
on manifolds will be discussed in Section \ref{sec2}.
\begin{theorem}\label{eigen}
Let $(M,g,f)$ be an $n$-dimensional closed gradient shrinking Ricci soliton
satisfying \eqref{Eq1} and \eqref{Eq2}. Let $\{0<\lambda_1\leq\lambda_2\leq\ldots\}$
be the set of eigenvalues of the Schr\"odinger operator $-\Delta+a\mathrm{R}$ with
$a\geq \frac 14$. Then
\[
\lambda_k\ge\frac{2n\pi}{e}\left(\frac{k\,e^\mu}{V(M)}\right)^{2/n}
\]
for all $k\geq1$, where $V(M)$ is the volume of $M$ and $\mu$ is Perelman's entropy
functional.
\end{theorem}
\begin{remark}
From the proof of Theorem \ref{eigen} in Section \ref{sec6}, we will see that
we can apply the same method to obtain similar eigenvalue estimates to allow
the compact gradient shrinking soliton to have convex boundaries with either
Dirichlet or Neumann boundary conditions.
\end{remark}

\begin{remark}\label{eigenrem}
For a bounded domain $\Omega\subset\mathbb{R}^n$, the well-known Weyl's asymptotic formula
of the $k$-th Dirichlet eigenvalue of the Laplace operator satisfies
\[
\lambda_k(\Omega)\sim c(n)\left(\frac{k}{V(\Omega)}\right)^{2/n},\quad k\to\infty,
\]
where $c(n)$ is the Weyl constant with $c(n)=4\pi^2\omega^{-2/n}_{n}$, $\omega_n$ is the volume of the
unit ball in $\mathbb{R}^n$. This indicates our lower eigenvalue estimates are sharp
for the exponent $2/n$ (at this time $\mathrm{R}=\mu=0$). Moreover the constant
$\frac{2n\pi}{e}\le c(n)$ and has the asymptotic property
\[
\lim_{n\to\infty}\frac{2n\pi}{e\cdot c(n)}=1.
\]
\end{remark}

We remark that Li and Yau \cite{[LYa]} used the Fourier transform method to get lower
bounds for Dirichlet eigenvalues of
the Laplace operator on a bounded domain $\Omega\subset \mathbb{R}^n$, which was later
generalized by them \cite{[LY]} to manifolds with Ricci curvature bounded below.
Grigor'yan \cite{[Gr2]} proved lower bounds for eigenvalues of the Schr\"odinger operator
under some assumption on the first eigenvalue. The author and P. Wu \cite{[WuW]}
obtained lower estimates for eigenvalues of the Witten-Laplace operator on compact
weighted manifolds.

The paper is organized as follows. In Section \ref{sec2}, we recall some basic
properties of Schr\"odinger heat kernels. We also introduce some identities
and a Logarithmic Sobolev inequality \cite{[LiWa]} on gradient shrinking Ricci
solitons. In Section \ref{sec3}, we apply the Logarithmic Sobolev inequality to
prove the ultracontractivity of Schr\"odinger heat kernels. In Section \ref{sec4},
we will prove Theorem \ref{Mainup2}. In Section \ref{sec5}, we apply Theorem
\ref{Mainup2} to derive an upper bound for the Green's function of the
Schr\"odinger operator. In Section \ref{sec6}, for compact gradient shrinking
Ricci solitons, we apply upper bounds for the Schr\"odinger heat kernel to
give eigenvalue estimates of the Schr\"odinger operator.

\section{Preliminaries}\label{sec2}
In this section, we summarize some basic facts about Schr\"odinger heat kernels
and gradient shrinking Ricci solitons. First we recall some basic results
regarding the Schr\"odinger heat kernel on manifolds. According to
Theorem 24.40 of \cite{[Choweta]}, we have the existence of the Schr\"odinger
heat kernel on manifolds.
\begin{theorem}\label{exist}
Let $(M,g)$ be a complete Riemannian manifold. If a given smooth function $Q(x)$
on manifold $(M,g)$ is bounded, then there exists a unique smooth Schr\"odinger
heat kernel $H^{Q}(x,y,t)$ for the operator $-\Delta+Q$.
\end{theorem}

Similar to the classical heat kernel case, the Schr\"odinger heat kernel
$H^{Q}(x,y,t)$  can be regarded as the limit of the Dirichlet Schr\"odinger
heat kernels on a sequence of exhausting subsets in $M$; see \cite{[Choweta]}.
The idea of the proof is as follows. Let $\Omega_1\subset\Omega_2\subset...\subset M$
be an exhaustion of relatively compact domains with smooth boundary in $(M,g)$.
In each $\Omega_k$, $k=1,2...$, we can construct the Dirichlet heat kernel
$H_{\Omega_k}^{Q}(x,y,t)$ for the operator $-\Delta+Q$.
By the maximum principle we have
\[
0< H_{\Omega_k}^{Q}\le H_{\Omega_{k+1}}^{Q},
\]
and
\[
\int_{\Omega_k}H_{\Omega_k}^{Q}(x,y,t)dv(x)\leq1.
\]
Therefore
\[
H^{Q}(x,y,t):=\lim_{k\to\infty}H_{\Omega_k}^{Q}(x,y,t)
\]
exists in $L^1(M)$ for any $(x,y,t)\in M\times M\times\mathbb{R}^+$. Following
the argument of Chapter VIII of \cite{[Ch]},  we can show that the limit
$H^{Q}(x,y,t)$ is finite, and is a smooth minimal positive fundamental solution
to the heat-type equation $\partial_t u-\Delta u+Q u=0$ on $M$. Moreover, the
Schr\"odinger heat kernel satisfies the symmetry property
\[
H_{\Omega_k}^{Q}(x,y,t)=H_{\Omega_k}^{Q}(y,x,t),\quad  H^{Q}(x,y,t)=H^{Q}(y,x,t)
\]
and the semigroup identity
\[
H_{\Omega_k}^{Q}(x,y,t+s)
=\int_{\Omega_k}H_{\Omega_k}^{Q}(x,z,t)H_{\Omega_k}^{Q}(z,y,s)dv(z),
\]
\[
H^{Q}(x,y,t+s)=\int_MH^{Q}(x,z,t)H^{Q}(z,y,s)dv(z).
\]

Since $Q$ is bounded, the Schr\"odinger operator $-\Delta+Q$ is self-adjoint
and its spectrum shares similar properties of the Laplace operator case
(see \cite{[DS2]}). For a compact subdomain $\Omega\subset M$, by the elliptic theory we let
$\{\varphi_k\}^{\infty}_{k=0}$ in $L^2(\Omega)$ be the complete orthonormal sequence of
the Dirichlet eigenfunctions of the operator $-\Delta+Q$ with the corresponding
non-decreasing sequence of discrete eigenvalues $\{\lambda_k\}^{\infty}_{k=1}$ satisfying
$\{0<\lambda_1\leq\lambda_2\leq\ldots\}$. Then the Dirichlet Schr\"odinger heat kernel
of $-\Delta+Q$ has the eigenfunction expansion
\[
H_{\Omega}^{Q}(x,y,t)=\sum^{\infty}_{k=1}e^{-\lambda_kt}\varphi_k(x)\varphi_k(y).
\]
Clearly, this expansion can be used for closed manifold $M$, i.e., $\Omega=M$.

In particular, we consider heat kernels of the Schr\"odinger operator $-\Delta+a\mathrm{R}$
on $(M,g,f)$. By Chen's result \cite{[Chen]},
we know that the scalar curvature $\mathrm{R}$ is nonnegative. If $\mathrm{R}$ is
bounded from above by a constant, by Theorem \ref{exist}, then the Schr\"odinger heat kernel $H^{\mathrm{R}}(x,y,t)$ always uniquely exists on $(M,g, f)$. Meanwhile, the above-mentioned
semigroup identity, symmetry and spectrum properties remain true for $H^{\mathrm{R}}(x,y,t)$.

Next we explain why \eqref{Eq2} holds on gradient shrinking Ricci soliton \eqref{Eq1}.
Using \eqref{Eq1}, we get
\[
\mathrm{R}+\Delta f=\tfrac n2,
\]
and
\begin{equation}\label{condition}
C(g):=\mathrm{R}+|\nabla f|^2-(f+c)
\end{equation}
is a finite constant, where $c\in \mathbb{R}$ is a free parameter to be determined later
(see Chapter 27 in \cite{[Choweta2]}). Combining these equalities gives
\begin{equation}\label{identitycond}
2\Delta f-|\nabla f|^2+\mathrm{R}+(f+c)-n=-C(g).
\end{equation}

On an $n$-dimensional complete Riemannian manifold $(M,g)$,
Perelman's $\mathcal{W}$-entropy functional \cite{[Pe]} is defined by
\[
\mathcal{W}(g,\phi,\tau)
:=\int_M\Big[\tau\Big(|\nabla \phi|^2+\mathrm{R}\Big)+\phi-n\Big](4\pi\tau)^{-n/2}e^{-\phi}dv
\]
for some $\phi\in C^\infty(M)$ and $\tau>0$, when this entropy functional
is finite, and Perelman's $\mu$-entropy functional \cite{[Pe]} is defined by
\[
\mu(g,\tau):=\inf\Big\{\mathcal{W}(g,\phi,\tau)\Big|\phi\in C_0^\infty(M)\,\,\,\text{with}\,\,\, \int_M(4\pi\tau)^{-n/2}e^{-\phi}dv=1\Big\}.
\]
Carrillo and Ni \cite{[CaNi]} observed that the function $f+c$ is always a
minimizer of $\mu(g,1)$ on a complete (possibly non-compact) gradient shrinking
Ricci soliton $(M,g,f)$. Therefore, by \eqref{identitycond}, we have
\begin{equation}
\begin{aligned}\label{intbp}
\mu(g,1)=\mathcal{W}(g,f+c,1)
&:=\int_M\Big(|\nabla f|^2+\mathrm{R}+(f+c)-n\Big)(4\pi)^{-n/2}e^{-(f+c)}dv\\
&=\int_M\Big(2\Delta f-|\nabla f|^2+\mathrm{R}+(f+c)-n\Big)(4\pi)^{-n/2}e^{-(f+c)}dv\\
&=-C(g),
\end{aligned}
\end{equation}
where $c$ is a constant such that $\int_M(4\pi)^{-n/2}e^{-(f+c)}dv=1$.
Notice that the above integral formulas always hold (see \cite{[HaMu]}
for the detailed explanation). If $\mathrm{R}+|\nabla f|^2=f$, then we
deduce that $\mu(g,1)=c$ and $\int_M(4\pi)^{-n/2}e^{-f}dv=e^{\mu(g,1)}$
by using \eqref{condition} and \eqref{intbp}. Hence we get \eqref{Eq2} in
the introduction.

Finally, we introduce an important Logarithmic Sobolev inequality on gradient shrinking
Ricci solitons, which will be useful in our paper. By Carrillo-Ni's result \cite{[CaNi]},
Li and Wang \cite{[LiWa]} proved a sharp Logarithmic Sobolev inequality on gradient
shrinking Ricci solitons without any curvature condition.
\begin{lemma}\label{logsi}
Let $(M,g, f)$ be an $n$-dimensional complete gradient shrinking Ricci soliton satisfying \eqref{Eq1}
and \eqref{Eq2}. For each compactly supported locally Lipschitz function $\varphi$ with
$\int_\Omega\varphi^2dv=1$ and each number $\tau>0$,
\[
\int_\Omega\varphi^2\ln \varphi^2dv\leq\tau\int_\Omega\left(4|\nabla\varphi|^2+\mathrm{R}\varphi^2\right)dv
-\left[\mu+n+\frac n2\ln(4\pi\tau)\right],
\]
where $\mathrm{R}$ is the scalar curvature of $(M,g,f)$ and $\mu$ is Perelman's entropy functional.
\end{lemma}

If the scalar curvature is bounded from above by a constant, Lemma \ref{logsi} reduces
to the defected	Logarithmic Sobolev inequality on manifolds. Li and Wang \cite{[LiWa]}
used the Logarithmic Sobolev inequality to prove a Sobolev inequality on complete gradient
shrinking Ricci solitons. In this paper, we will apply the Logarithmic Sobolev inequality
to derive sharp Gaussian upper bounds for the Schr\"odinger heat kernel.


\section{Ultracontractivity}\label{sec3}

In this section, we will apply the Logarithmic Sobolev inequality (Lemma \ref{logsi}) to
derive upper bounds for the Schr\"odinger heat kernels on gradient shrinking
Ricci solitons. In other words, we will show that the Schr\"odinger heat kernel enjoys
the ultracontractivity. A similar result has been explored for heat kernels
of the Laplace operator on manifolds \cite{[Da],[Zh5]}.

\begin{theorem}\label{Mainup}
Let $(M,g, f)$ be an $n$-dimensional complete gradient shrinking Ricci soliton satisfying
\eqref{Eq1} and \eqref{Eq2} with scalar curvature $\mathrm{R}$ bounded above by a constant.
Then the Schr\"odinger heat kernel of the operator $-\Delta+a\mathrm{R}$ with $a\ge\frac 14$
satisfies
\begin{equation}\label{low}
H^{\mathrm{R}}(x,y,t)\le\frac{e^{-\mu}}{(4\pi t)^{\frac n2}}
\end{equation}
for all $x,y\in M$ and $t>0$, where $\mu$ is Perelman's entropy functional.
\end{theorem}
\begin{remark}
As in Remark \ref{remex}, the assumption on scalar curvature here is only needed
to guarantee the existence and semigroup property of the Schr\"odinger heat kernel
$H^{\mathrm{R}}(x,y,t)$ on shrinking Ricci solitons and it is not directly used
in the proof of Theorem \ref{Mainup}.
\end{remark}

\begin{remark}
Recently, Li and Wang \cite{[LiWa]} proved a similar upper bound for the conjugate heat
kernel on ancient solutions of the Ricci flow induced by a Ricci shrinking soliton.
For a fixed-metric Ricci shrinking soliton, the study of the Schr\"odinger heat
kernel seems to be more reasonable than the evolved-metric setting.
\end{remark}

\begin{proof}[Proof of Theorem \ref{Mainup}]
By the approximation argument, it suffices to prove \eqref{low} for the Dirichlet
Schr\"odinger heat kernel $H_{\Omega}^\mathrm{R}(x,y,t)$ of any compact set $\Omega$
in $(M,g,f)$. In fact, let $\Omega_i$, $i=1,2,...$,  be a compact exhaustion of $M$
such that $\Omega_i\subset\Omega_{i+1}$ and $\cup_i\Omega_i=M$.
If we are able to prove \eqref{low} for the Dirichlet Schr\"odinger heat kernel
$H_{\Omega_i}^\mathrm{R}(x,y,t)$ for any $i$, then the result follows by letting
$i\to\infty$.

We will use the argument of \cite{[Da]} (see also \cite{[Zh5]}) to give the estimate
\eqref{low}. Let $u=u(x,t)$, $t\in[0,T]$, be a smooth solution to the heat-type Schr\"odinger
equation
\[
(\partial_t+L)u=0,
\]
where $L=-\Delta+a\mathrm{R}$, in a compact set $\Omega\subset M$ with Dirichlet boundary
condition: $u(x,t)=0$ on $\partial\Omega$. Then $u(x,t)$ can be written as
\[
u(x,t)=\int_\Omega u(y,0)H_{\Omega}^\mathrm{R}(x,y,t)dv(y),
\]
where $H_{\Omega}^\mathrm{R}(x,y,t)$ denotes the Schr\"odinger heat kernel of the operator
$L$ in the compact set $\Omega\subset M$.

In the following, we shall estimate
\[
\|u\|_{p(t)}:=\left(\int_\Omega|u|^{p(t)}dv\right)^{\frac{1}{p(t)}},
\]
where $p(t)=\frac{T}{T-t}$, $t\in[0,T]$, which obviously satisfies $p(0)=1$ and $p(T)=\infty$.
From Lemma 2.2.2 of \cite{[Da]}, we know that $\|u\|_{p(t)}$ is a continuously differentiable
function of $t$. Therefore, we compute
\begin{equation*}
\begin{aligned}
\partial_t\|u\|_{p(t)}&=-\frac{p'(t)}{p^2(t)}\|u\|_{p(t)}\cdot\ln\left(\|u\|^{p(t)}_{p(t)}\right)\\
&\quad+\frac{\|u\|^{1-p(t)}_{p(t)}}{p(t)}\left[p'(t)\int_\Omega u^{p(t)}\ln udv+p(t)\int_\Omega u^{p(t)-1}u_tdv\right]\\
&=-\frac{p'(t)}{p^2(t)}\|u\|_{p(t)}\cdot\ln\left(\|u\|^{p(t)}_{p(t)}\right)\\
&\quad+\frac{\|u\|^{1-p(t)}_{p(t)}}{p(t)}\left[p'(t)\int_\Omega u^{p(t)}\ln udv+p(t)\int_\Omega u^{p(t)-1}\Delta udv-a
p(t)\int_\Omega\mathrm{R}udv\right].
\end{aligned}
\end{equation*}
Multiplying by $p^2(t)\|u\|^{p(t)}_{p(t)}$ in the above equality and integrating by parts for the
term $\Delta u$, we have
\begin{equation*}
\begin{aligned}
p^2(t)\|u\|^{p(t)}_{p(t)}\cdot\partial_t\|u\|_{p(t)}
&=-p'(t)\|u\|^{1+p(t)}_{p(t)}\cdot\ln\left(\|u\|^{p(t)}_{p(t)}\right)+p(t)p'(t)\|u\|_{p(t)}\int_\Omega u^p\ln udv\\
&\quad-p^2(t)(p(t)-1)\|u\|_{p(t)}\int_\Omega u^{p(t)-2}|\nabla u|^2dv-ap^2(t)\|u\|_{p(t)}\int_\Omega\mathrm{R}u^pdv.
\end{aligned}
\end{equation*}
Dividing by $\|u\|_{p(t)}$ in the above equality yields
\begin{equation*}
\begin{aligned}
p^2(t)\|u\|^{p(t)}_{p(t)}\cdot\partial_t\left(\ln\|u\|_{p(t)}\right)
&=-p'(t)\|u\|^{p(t)}_{p(t)}\cdot\ln\left(\|u\|^{p(t)}_{p(t)}\right)+p(t)p'(t)\int_\Omega u^p\ln udv\\
&\quad-4(p(t)-1)\int_\Omega|\nabla u^{\frac{p(t)}{2}}|^2dv-ap^2(t)\int_\Omega\mathrm{R}u^pdv,
\end{aligned}
\end{equation*}
which further implies
\begin{equation}
\begin{aligned}\label{equa}
p^2(t)\cdot\partial_t\left(\ln\|u\|_{p(t)}\right)
&=-p'(t)\cdot\ln\left(\|u\|^{p(t)}_{p(t)}\right)+\frac{p(t)p'(t)}{\|u\|^{p(t)}_{p(t)}}\int_\Omega u^p\ln udv\\
&\quad-\frac{4(p(t)-1)}{\|u\|^{p(t)}_{p(t)}}\int_\Omega|\nabla
u^{\frac{p(t)}{2}}|^2dv-\frac{ap^2(t)}{\|u\|^{p(t)}_{p(t)}}\int_\Omega\mathrm{R}u^pdv.
\end{aligned}
\end{equation}
We now introduce a new quantity to simplify equality \eqref{equa}. Set
\[
w:=\frac{u^{\frac{p(t)}{2}}}{\|u^{\frac{p(t)}{2}}\|_2}.
\]
Then, we see that
\[
w^2=\frac{u^{p(t)}}{\|u\|^{p(t)}_{p(t)}},\quad \|w\|_2=1
\quad \text{and}\quad \ln w^2=\ln u^{p(t)}-\ln\left(\|u\|^{p(t)}_{p(t)}\right).
\]
So, we have
\begin{equation*}
\begin{aligned}
p'(t)\int_\Omega w^2\ln w^2dv&=p'(t)\int_\Omega\frac{u^{p(t)}}{\|u\|^{p(t)}_{p(t)}}\left[\ln
u^{p(t)}-\ln\left(\|u\|^{p(t)}_{p(t)}\right)\right]dv\\
&=\frac{p(t)p'(t)}{\|u\|^{p(t)}_{p(t)}}\int_\Omega u^{p(t)}\ln udv-p'(t)\ln\left(\|u\|^{p(t)}_{p(t)}\right).
\end{aligned}
\end{equation*}
Using the above equality, \eqref{equa} can be simplified as
\begin{equation}
\begin{aligned}\label{equa2}
p^2(t)\partial_t\left(\ln\|u\|_{p(t)}\right)
&=p'(t)\int_\Omega w^2\ln w^2dv-4(p(t)-1)\int_\Omega|\nabla w|^2dv-ap^2(t)\int_\Omega\mathrm{R}w^2dv\\
&=p'(t)\left[\int_\Omega w^2\ln w^2dv-\frac{4(p(t)-1)}{p'(t)}\int_\Omega|\nabla
w|^2dv-\frac{ap^2(t)}{p'(t)}\int_\Omega\mathrm{R}w^2dv\right].
\end{aligned}
\end{equation}
Now we want to apply the Logarithmic Sobolev inequality (Lemma \ref{logsi}) to
estimate \eqref{equa2}. Indeed, if we choose
\[
\varphi=w \quad \text{and}\quad 4\tau=\frac{4(p(t)-1)}{p'(t)}=\frac{4t(T-t)}{T}\leq T
\]
in Lemma \ref{logsi}, then this gives
\[
\int_\Omega w^2\ln w^2dv\leq\frac{(p(t)-1)}{p'(t)}\int_\Omega\left(4|\nabla w|^2+\mathrm{R}w^2\right)dv
-\left[\mu+n+\frac n2\ln(4\pi\tau)\right].
\]
Using this, \eqref{equa2} can be reduced to
\[
p^2(t)\partial_t\left(\ln\|u\|_{p(t)}\right)
\le p'(t)\left[\frac{p(t)-1-ap^2(t)}{p'(t)}\int_\Omega\mathrm{R}w^2dv-\mu-n-\frac n2\ln(4\pi\tau)\right].
\]
Since the scalar curvature $\mathrm{R}\ge 0$ on $(M,g,f)$ due to Chen \cite{[Chen]} and
\begin{equation*}
\begin{aligned}
p(t)-1-ap^2(t)&=-a\left(p(t)-\frac{1}{2a}\right)^2+\left(\frac{1}{4a}-1\right)\\
&\le0,
\end{aligned}
\end{equation*}
where we used $a\geq\frac 14$ in the second inequality above, then
\[
p^2(t)\partial_t\left(\ln\|u\|_{p(t)}\right)
\le p'(t)\left[-\mu-n-\frac n2\ln(4\pi\tau)\right].
\]
Noticing that
\[
\frac{p'(t)}{p^2(t)}=\frac 1T\quad \text{and} \quad\tau=\frac{t(T-t)}{T},
\]
then we obtain
\[
\partial_t\left(\ln\|u\|_{p(t)}\right)
\le \frac 1T\left[-\mu-n-\frac n2\ln\frac{4\pi t(T-t)}{T}\right].
\]
Integrating the above inequality from $0$ to $T$ with respect to $t$, we have
\[
\ln\left(\frac{\|u(x,T)\|_{p(T)}}{\|u(x,0)\|_{p(0)}}\right)
\le -\mu-\frac n2\ln(4\pi)-\frac n2\ln T.
\]
Notice that $p(0)=1$ and $p(T)=\infty$, and we have
\[
\|u(x,T)\|_{\infty}\le\|u(x,0)\|_1\cdot\frac{e^{-\mu}}{(4\pi T)^{\frac n2}}.
\]
Since
\[
u(x,T)=\int_\Omega u(y,0)H_{\Omega}^\mathrm{R}(x,y,T)dv(y),
\]
then we conclude
\begin{equation}\label{CHUP}
H_{\Omega}^\mathrm{R}(x,y,T)\le\frac{e^{-\mu}}{(4\pi T)^{\frac n2}}
\end{equation}
and the result follows since $T$ is arbitrary.
\end{proof}

By a similar argument, when $a=0$, we also have
\begin{proposition}\label{corup}
Let $(M,g, f)$ be an $n$-dimensional complete gradient shrinking Ricci soliton
satisfying \eqref{Eq1} and \eqref{Eq2}. If the scalar curvature $\mathrm{R}$ of $(M,g, f)$
satisfies
\[
\mathrm{R}\le C_R
\]
for some constant $C_R\ge 0$, then the heat kernel $H(x,y,t)$
of the Laplace operator satisfies
\begin{equation}\label{corlow}
H(x,y,t)\leq\frac{e^{-\mu}}{(4\pi t)^{\frac n2}}\exp\left(\frac{C_R\,t}{6}\right)
\end{equation}
for all $x,y\in M$ and $t>0$, where $\mu$ is Perelman's entropy functional.
\end{proposition}

In the end of this section, by using the argument of Varopoulos \cite{[Varo2]}, we
can apply Theorem \ref{Mainup} to give a Sobolev inequality proved by Li and Wang
(see Corollary 5.13 in \cite{[LiWa]}) on complete gradient shrinking Ricci solitons.
Here we only provide the result without proof. The detailed proof could follow the
argument of Theorem 11.6 in \cite{[Lip]} by using the Schr\"odinger operator $L$
instead of the Laplace operator.
\begin{proposition}
Let $(M,g, f)$ be an $n$-dimensional complete gradient shrinking Ricci soliton satisfying
\eqref{Eq1} and \eqref{Eq2} with scalar curvature $\mathrm{R}$ bounded above by a constant.
Then there exists a constant depending only on $n$ such that
\[
\left(\int_{B_r(p)} u^{\frac{2n}{n-2}}\,dv\right)^{\frac{n-2}{n}}\le C(n)e^{-\frac{2\mu}{n}} \int_{B_r(p)}\left(|\nabla u|^2+a\mathrm{R}\,u^2\right) dv
\]
for each compactly supported smooth function $u$ with supported in a geodesic ball
$B_r(p)$ of radius $r$ with center at $p\in M$. Here $\mu:=\mu(g,1)$ is Perelman's
entropy functional, and $a$ is a constant with $a\ge\frac 14$.
\end{proposition}


\section{Gaussian upper bound}\label{sec4}
In this section, we will follow the argument of Grigor'yan \cite{[Gr1]} to prove
Theorem \ref{Mainup2}. Let $(M,g, f)$ be a gradient shrinking Ricci soliton.
For a pre-compact region $\Omega\subset M$ and a compact set $K\subset\Omega$,
let $u(x,t)$ be a smooth solution to the Dirichlet problem for the equation
$(\partial_t+L)u=0$ in $\Omega\times(0,T)$ (with an initial condition having
a support on $K$), where $L=-\Delta+a\mathrm{R}$. For such a solution $u(x,t)$, 
we consider two integrals
\[
I(t):=\int_\Omega u^2(x,t)dv
\]
and
\[
E_D(t):=\int_\Omega u^2(x,t)\exp\left(\frac{d^2(x,K)}{Dt}\right)dv,
\]
where $D$ is a positive number. Obviously, $I(t)\le E_D(t)$. In the following, we will prove
a reverse inequality in some ways.

\begin{proposition}\label{pro1}
Let $u(x,t)$ be a smooth solution to the Dirichlet problem for the equation $(\partial_t+L)u=0$
on $(M,g, f)$. Assume that for any $t\in(0,T)$,
\begin{equation}\label{Iint}
I(t)\le \frac{e^{-\mu}}{(8\pi t)^{\frac n2}}.
\end{equation}
Then, for any $\gamma>1$, $D>2$ and for all $t\in(0,T)$,
\[
E_D(t)\leq \frac{4e^{-\mu}}{(8\pi\delta t)^{\frac n2}}
\]
for some $\delta=\delta(D,\gamma)>0$. Here $\mu$ is Perelman's entropy functional.
\end{proposition}

In order to prove this proposition, we start from a useful lemma.
\begin{lemma}\label{lem1}
Under the hypotheses of Proposition \ref{pro1}, for any $\gamma>1$,
there exists $D_0=D_0(\gamma)>2$, such that
\[
I_R(t)\leq \frac{2e^{-\mu}}{(8\pi \frac{t}{\gamma})^{\frac n2}}\exp\left(-\frac{R^2}{D_0t}\right)
\]
for all $R>0$ and $t\in(0,T)$, where $\mu$ is Perelman's entropy functional, and
\[
I_R(t):=\int_{\Omega\setminus B(K,R)}u^2(x,t)dv.
\]
Here $B(K,R)$ denotes the open $R$-neighbourhood of the set $K$.
\end{lemma}

\begin{proof}[Proof of Lemma \ref{lem1}]
To prove the estimate, we first claim that $I_R(t)$ satisfies a comparison result:
\begin{equation}\label{claim1}
I_R(t)\leq I_r(\tau)+\frac{e^{-\mu}}{(8\pi\tau)^{\frac n2}}\exp\left(-\frac{(R-r)^2}{2(t-\tau)}\right)
\end{equation}
for $R>r$ and $t>\tau$. This claim follows by an integral monotonicity, which says that
the following function
\[
\int_\Omega u^2(x,t)e^{\xi(x,t)}dv,
\]
is non-increasing in $t\in(0,T)$. Here the function $\xi(x,t)$ is defined as
\[
\xi(x,t):=\frac{d^2(x)}{2(t-s)}
\]
for $s>t$, where $d(x)$ is a distance function defined by
\begin{equation*}
d(x)=\left\{ \begin{aligned}
&R-d(x,K)&&\mathrm{if}\,\,x\in B(K,R), \\
&0&&\mathrm{if}\,\,x\notin B(K,R). \\
\end{aligned}\right.
\end{equation*}
By Lemma 3.3 in \cite{[Ta]}, we know that $\int_\Omega u^2(x,t)e^{\xi(x,t)}dv$
is an almost everywhere differentiable function of $t$. So its monotonicity could
be obtained by the direct computation
\begin{equation*}
\begin{aligned}
\frac{d}{dt}\int_\Omega u^2(x,t)e^{\xi(x,t)}dv&=\int_\Omega u^2\xi_te^{\xi}dv+\int_\Omega 2uu_te^{\xi}dv\\
&\leq-\frac 12\int_\Omega u^2|\nabla\xi|^2e^{\xi}dv+\int_\Omega 2u(\Delta u-a\mathrm{R}u)e^{\xi}dv\\
&\leq-\frac 12\int_\Omega u^2|\nabla\xi|^2e^{\xi}dv-2\int_\Omega \nabla u\nabla(ue^{\xi})dv\\
&=-\frac 12\int_\Omega (u\nabla\xi+2\nabla u)^2e^\xi dv\\
&\leq0,
\end{aligned}
\end{equation*}
where we used the scalar curvature $\mathrm{R}\ge 0$ due to Chen \cite{[Chen]}
in the above inequality. We now continue to prove the claim
\eqref{claim1}. By the integral monotonicity, we have
\begin{equation}\label{ineq1}
\int_\Omega u^2(x,t)e^{-\frac{d^2(x)}{2(s-t)}}dv\le\int_\Omega u^2(x,\tau)e^{-\frac{d^2(x)}{2(s-\tau)}}dv
\end{equation}
for $s>t>\tau$. Notice that, by the definition of $d(x)$, on one hand,
\begin{equation*}
\begin{aligned}
\int_\Omega u^2(x,t)e^{-\frac{d^2(x)}{2(s-t)}}dv
&=\int_{\Omega\setminus B(K,R)} u^2(x,t)e^{-\frac{d^2(x)}{2(s-t)}}dv
+\int_{B(K,R)} u^2(x,t)e^{-\frac{d^2(x)}{2(s-t)}}dv\\
&\ge \int_{\Omega\setminus B(K,R)}u^2(x,t)dv\\
&=I_R(t);
\end{aligned}
\end{equation*}
on the other hand,
\begin{equation*}
\begin{aligned}
\int_\Omega u^2(x,\tau)e^{-\frac{d^2(x)}{2(s-\tau)}}dv
&=\int_{\Omega\setminus B(K,r)} u^2(x,\tau)e^{-\frac{d^2(x)}{2(s-\tau)}}dv
+\int_{B(K,r)} u^2(x,\tau)e^{-\frac{d^2(x)}{2(s-\tau)}}dv\\
&\le\int_{\Omega\setminus B(K,r)} u^2(x,\tau)dv
+\int_{B(K,r)} u^2(x,\tau)e^{-\frac{(R-r)^2}{2(s-\tau)}}dv\\
&=I_r(t)+\exp\left(-\frac{(R-r)^2}{2(s-\tau)}\right)\int_{B(K,r)} u^2(x,\tau)dv,
\end{aligned}
\end{equation*}
where $R>r$. Combining these estimates, \eqref{ineq1} becomes
\[
I_R(t)\le I_r(\tau)+\exp\left(-\frac{(R-r)^2}{2(s-\tau)}\right)\int_{B(K,r)} u^2(x,\tau)dv,
\]
and hence claim \eqref{claim1} follows by letting $s\to t+$ and the assumption of Proposition \ref{pro1}.

Then, we will apply \eqref{claim1} to prove Lemma \ref{lem1} by some iteration technique.
Choose $R_k$ and $t_k$ as follows:
\[
R_k=\left(\frac 12+\frac{1}{k+2}\right)R,\quad\quad t_k=\frac{t}{\gamma^k},
\]
where $\gamma>1$ is a fixed constant. We apply \eqref{claim1} to pairs $(R_k,t_k)$ and
$(R_{k+1},t_{k+1})$ and get the following iterated inequality
\[
I_{R_k}(t_k)\leq I_{R_{k+1}}(t_{k+1})+\frac{e^{-\mu}}{(8\pi t_{k+1})^{\frac
n2}}\exp\left(-\frac{(R_k-R_{k+1})^2}{2(t_k-t_{k+1})}\right).
\]
Sum up the above inequalities over all $k=0,1,2...$,
\[
I_R(t)\le \sum^{\infty}_{k=0}\frac{e^{-\mu}}{(8\pi t_{k+1})^{\frac
n2}}\exp\left[-\frac{(R_k-R_{k+1})^2}{2(t_k-t_{k+1})}\right],
\]
where we used the fact that
\[
\lim_{k\to\infty}I_{R_k}(t_k)=\int_{\Omega\setminus B(K,R/2)}u^2(x,0)dv=0
\]
by the Dirichlet boundary condition of $u$. Since
\[
t_{k+1}=\frac{t}{\gamma^{k+1}},\quad R_k-R_{k+1}\ge\frac{R}{(k+3)^2}\quad\text{and}\quad
t_k-t_{k+1}=\frac{\gamma-1}{\gamma^{k+1}}t,
\]
then we have
\[
I_R(t)\le\frac{e^{-\mu}}{(8\pi t)^{\frac n2}}\sum^{\infty}_{k=0}\exp\left[(k+1)\frac n2\ln
\gamma-\frac{\gamma^{k+1}}{(\gamma-1)(k+3)^4}\cdot\frac{R^2}{2t}\right].
\]
Notice that $\gamma^{k+1}$ grows in $k$ much faster than the
denominator $(k+3)^4$ whenever $\gamma>1$. So there exists a positive
number $m=m(\gamma)<1$ such that
\begin{equation}\label{num}
\frac{\gamma^{k+1}}{(\gamma-1)(k+3)^4}\ge m(k+2)
\end{equation}
for any $k\ge 0$. In particular, we can take
\[
m=m(\gamma):=\min\left\{\inf_{k\geq0}\frac{\gamma^{k+1}}{(\gamma-1)(k+2)(k+3)^4},\,\, \frac{3}{4}\right\}.
\]
Then,
\begin{equation*}
\begin{aligned}
I_R(t)&\le\frac{e^{-\mu}}{(8\pi t)^{\frac n2}}\sum^{\infty}_{k=0}\exp\left[(k+1)\frac n2\ln
\gamma-m(k+2)\cdot\frac{R^2}{2t}\right]\\
&=\frac{e^{-\mu}}{(8\pi t)^{\frac n2}}\exp\left(-m\frac{R^2}{2t}\right)\sum^{\infty}_{k=0}\exp\left[(k+1)\left(\frac
n2\ln \gamma-m\frac{R^2}{2t}\right)\right].
\end{aligned}
\end{equation*}

We shall further estimate the right-hand side of the above inequality.
When
\[
\frac n2\ln \gamma-m\frac{R^2}{2t}\leq-\ln 2,
\]
we have
\begin{equation*}
\begin{aligned}
I_R(t)&\le\frac{e^{-\mu}}{(8\pi t)^{\frac n2}}\exp\left(-m\frac{R^2}{2t}\right)\sum^{\infty}_{k=0}2^{-(k+1)}\\
&=\frac{e^{-\mu}}{(8\pi t)^{\frac n2}}\exp\left(-m\frac{R^2}{2t}\right).
\end{aligned}
\end{equation*}
When
\[
\frac n2\ln\gamma-m\frac{R^2}{2t}>-\ln2,
\]
we use the definitions of $I_R(t)$ and $I(t)$ and have that
\begin{equation*}
\begin{aligned}
I_R(t)\le I(t)&\le\frac{e^{-\mu}}{(8\pi t)^{\frac n2}}\\
&\le\frac{e^{-\mu}}{(8\pi t)^{\frac n2}}\exp\left(\frac n2\ln\gamma+\ln2-m\frac{R^2}{2t}\right)\\
&=\frac{2e^{-\mu}}{(8\pi \frac{t}{\gamma})^{\frac n2}}\exp\left(-m\frac{R^2}{2t}\right).
\end{aligned}
\end{equation*}
Therefore, in any case
\[
I_R(t)\le\frac{2e^{-\mu}}{(8\pi \frac{t}{\gamma})^{\frac n2}}\exp\left(-m\frac{R^2}{2t}\right),
\]
where $m=m(\gamma)<1$ and Lemma \ref{lem1} follows.
\end{proof}

Now we apply Lemma \ref{lem1} to give the proof of Proposition \ref{pro1}.
\begin{proof}[Proof of Proposition \ref{pro1}]
\emph{Step One}: we show that for $D\ge5D_0$ and for all $t>0$,
\[
E_D(t)\leq\frac{4e^{-\mu}}{(8\pi \frac{t}{\gamma})^{\frac n2}}.
\]
By the definition of $E_D(t)$, we split $E_D(t)$ into two terms:
\begin{equation*}
\begin{aligned}
E_D(t)&=\int_\Omega u^2\exp\left(\frac{d^2(x,K)}{Dt}\right)dv\\
&=\int_{\{d(x,K)\le R\}} u^2\exp\left(\frac{d^2(x,K)}{Dt}\right)dv
+\sum^{\infty}_{k=0}\int_{\{2^kR\le d(x,K)\le 2^{k+1}R\}} u^2\exp\left(\frac{d^2(x,K)}{Dt}\right)dv\\
&\le\int_{\Omega} u^2\exp\left(\frac{R^2}{Dt}\right)dv
+\sum^{\infty}_{k=0}\int_{\{2^kR\le d(x,K)\le 2^{k+1}R\}} u^2\exp\left(\frac{d^2(x,K)}{Dt}\right)dv\\
&\le\frac{e^{-\mu}}{(8\pi t)^{\frac n2}}\exp\left(\frac{R^2}{Dt}\right)
+\sum^{\infty}_{k=0}\int_{\{2^kR\le d(x,K)\le 2^{k+1}R\}} u^2\exp\left(\frac{d^2(x,K)}{Dt}\right)dv.
\end{aligned}
\end{equation*}
Since we have the condition \eqref{Iint}, by Lemma \ref{lem1}, the $k$-th factor in the above
sum term can be estimated by
\begin{equation*}
\begin{aligned}
\int_{\{2^kR\le d(x,K)\le 2^{k+1}R\}} u^2\exp\left(\frac{d^2(x,K)}{Dt}\right)dv
&\le\exp\left(\frac{4^{k+1}R^2}{Dt}\right)\int_{\Omega\setminus B(K,2^kR)} u^2dv\\
&\le\frac{2e^{-\mu}}{(8\pi \frac{t}{\gamma})^{\frac n2}}\exp\left(\frac{4^{k+1}R^2}{Dt}-\frac{4^kR^2}{D_0t}\right)\\
&\le\frac{2e^{-\mu}}{(8\pi \frac{t}{\gamma})^{\frac n2}}\exp\left(-\frac{4^kR^2}{Dt}\right),
\end{aligned}
\end{equation*}
where we used $D\ge5D_0$. Therefore,
\[
E_D(t)\le\frac{e^{-\mu}}{(8\pi t)^{\frac n2}}\exp\left(\frac{R^2}{Dt}\right)
+\frac{2e^{-\mu}}{(8\pi \frac{t}{\gamma})^{\frac n2}}\sum^{\infty}_{k=0}\exp\left(-\frac{4^kR^2}{Dt}\right)
\]
for any $R>0$. In particular, we choose $R^2=Dt\ln 2$ and get
\begin{equation*}
\begin{aligned}
E_D(t)&\le\frac{2e^{-\mu}}{(8\pi t)^{\frac n2}}
+\frac{2e^{-\mu}}{(8\pi \frac{t}{\gamma})^{\frac n2}}\sum^{\infty}_{k=0}2^{-4^k}\\
&\le\frac{4e^{-\mu}}{(8\pi \frac{t}{\gamma})^{\frac n2}}.
\end{aligned}
\end{equation*}

\emph{Step Two}:  In the rest, it suffices to prove the case $2<D<5D_0$.
Similar to the preceding discussion in Lemma \ref{lem1}, we also
claim that the integral
\[
\int_\Omega u^2(x,t)e^{\frac{d^2(x,K)}{2(t+s)}}dv
\]
is non-increasing in $t\in(0,\infty)$ for any $s>0$. Since this integral quantity
is almost everywhere differentiable with respect to $t$, setting
$\eta=\eta(x,t):=\frac{d^2(x,K)}{2(t+s)}$, then
\begin{equation*}
\begin{aligned}
\frac{d}{dt}\int_\Omega u^2(x,t)e^{\frac{d^2(x,K)}{2(t+s)}}dv
&=\int_\Omega u^2\eta_te^{\eta}dv+\int_\Omega 2uu_te^{\eta}dv\\
&\leq-\frac{1}{2}\int_\Omega u^2|\nabla\eta|^2 e^{\eta}dv+\int_\Omega 2u(\Delta u-a\mathrm{R}u)e^{\eta}dv\\
&\leq-\frac 12\int_\Omega u^2|\nabla\eta|^2e^{\eta}dv-2\int_\Omega \nabla u\nabla(ue^{\eta})dv\\
&=-\frac 12\int_\Omega (u\nabla\eta+2\nabla u)^2e^\eta dv\\
&\leq0,
\end{aligned}
\end{equation*}
where we also used $\mathrm{R}\ge 0$ on $(M,g,f)$, and the claim follows. Therefore, for any $\tau\in(0,t)$,
\[
\int_\Omega u^2(x,t)e^{\frac{d^2(x,K)}{2(t+s)}}dv\le\int_\Omega u^2(x,\tau)e^{\frac{d^2(x,K)}{2(\tau+s)}}dv.
\]
Given $2<D<5D_0$ and $t$, we let $s=\frac{D-2}{2}t$ and $\tau=\frac{D-2}{5D_0-2}t<t$ in the above inequality, which thus rewrites as
\[
E_D(t)\le E_{5D_0}(\tau)
\]
for any $\tau\in(0,t)$. By step one, we have proven
\[
E_{5D_0}(\tau)\le \frac{4e^{-\mu}}{(8\pi \frac{\tau}{\gamma})^{\frac n2}}.
\]
Hence
\[
E_D(t)\le \frac{4e^{-\mu}}{\left(\frac{8\pi(D-2)}{5D_0-2}\cdot\frac{t}{\gamma}\right)^{\frac n2}}.
\]
This implies Proposition \ref{pro1} by letting $\delta=\delta(D,\gamma)=\frac{D-2}{5D_0-2}\gamma^{-1}$.
\end{proof}

Next, we will apply Proposition \ref{pro1} to prove Theorem \ref{Mainup2}.
\begin{proof}[Proof of Theorem \ref{Mainup2}]
By the semigroup property of the Schr\"odinger heat kernel, we have
\[
H^{\mathrm{R}}(x,y,t)=\int_MH^{\mathrm{R}}(x,z,t/2)H^{\mathrm{R}}(z,y,t/2)dv(z).
\]
Then by the triangle inequality $d^2(x,y)\le2(d^2(x,z)+d^2(y,z))$, for any positive constant $D$,
we furthermore have
\begin{equation*}
\begin{aligned}
H^{\mathrm{R}}(x,y,t)&\le\int_MH^{\mathrm{R}}(x,z,t/2)e^{\frac{d^2(x,z)}{Dt}}H^{\mathrm{R}}(z,y,t/2)
e^{\frac{d^2(y,z)}{Dt}}e^{-\frac{d^2(x,y)}{2Dt}}dv(z)\\
&\le e^{-\frac{d^2(x,y)}{2Dt}}\left[\int_M\left(H^{\mathrm{R}}(x,z,t/2)e^{\frac{d^2(x,z)}{Dt}}\right)^2dv(z)\right]^{1/2}\\
&\quad\times \left[\int_M\left(H^{\mathrm{R}}(y,z,t/2)e^{\frac{d^2(y,z)}{Dt}}\right)^2dv(z)\right]^{1/2}.
\end{aligned}
\end{equation*}
If we set
\[
E_D(x,t):=\int_M (H^{\mathrm{R}}(x,z,t))^2e^{\frac{d^2(x,z)}{Dt}}dv(z),
\]
then we have a simple expression
\begin{equation}\label{inteuphet}
H^{\mathrm{R}}(x,y,t)\le\sqrt{E_D(x,t/2)E_D(y,t/2)}\exp\left(-\frac{d^2(x,y)}{2Dt}\right),
\end{equation}
which always holds on $(M,g,f)$.

We take an increasing sequence of pre-compact regions $\Omega_k\subset M$, $k\in \mathbb{N}$,
exhausting $M$, and in each $\Omega_k$ we construct the Dirichlet heat kernel
$H_{\Omega_k}^{\mathrm{R}}(x,y,t)$ for the heat-type equation $(\partial_t+L)u=0$.
Then by the maximum principle, we have
\[
0< H_{\Omega_k}^{\mathrm{R}}\le H_{\Omega_{k+1}}^{\mathrm{R}}\le H^{\mathrm{R}}.
\]
Now we will apply Proposition \ref{pro1} to estimate $E_{D,{\Omega_k}}(x,t)$, where
\[
E_{D,{\Omega_k}}(x,t):=\int_{\Omega_k} (H^{\mathrm{R}}(x,z,t))^2e^{\frac{d^2(x,z)}{Dt}}dv(z).
\]
Let $u(z,t)=H_{\Omega_k}^{\mathrm{R}}(x,z,t)$ be a solution to the Dirichlet problem
for the equation $(\partial_t+L)u=0$ in $\Omega_k\times(0,T)$ and let $K=\{x\}\subset\Omega$.
We observe that
\begin{equation*}
\begin{aligned}
I(t)&=\int_{\Omega_k} u^2(z,t)dv(z)\\
&=\int_{\Omega_k}H_{\Omega_k}^{\mathrm{R}}(x,z,t)H_{\Omega_k}^{\mathrm{R}}(z,x,t)dv(z)\\
&=H_{\Omega_k}^{\mathrm{R}}(x,x,2t)\\
&\le\frac{e^{-\mu}}{(8\pi t)^{\frac n2}},
\end{aligned}
\end{equation*}
where we used \eqref{CHUP} in the last inequality.
So we can apply Proposition \ref{pro1} to get that for any $\gamma>1$, $D>2$ and for all
$t\in(0,T)$, we have
\begin{equation*}
\begin{aligned}
E_{D,{\Omega_k}}(x,t)&=\int_{\Omega_k} u^2(z,t)e^{\frac{d^2(x,z)}{Dt}}dv(z)\\
&\le \frac{4e^{-\mu}}{(8\pi\delta t)^{\frac n2}}
\end{aligned}
\end{equation*}
for some $\delta=\delta(D,\gamma)>0$. Indeed we can take
$\delta=\delta(D,\gamma)=\frac{D-2}{5D_0-2}\gamma^{-1}$, where $2<D<\frac{10}{m(\gamma)}$
and $m(\gamma)<1$ is determined by \eqref{num}. Since the number $\delta$ does not
depend on $k$, letting $k\to\infty$ in the above inequality, we then have
\[
E_D(x,t/2)\le \frac{4e^{-\mu}}{(4\pi\delta t)^{\frac n2}}.
\]
By a similar argument to the point $y\in M$, we have
\[
E_D(y,t/2)\le \frac{4e^{-\mu}}{(4\pi\delta t)^{\frac n2}}.
\]
Substituting the above two estimates into \eqref{inteuphet} completes the proof of
Theorem \ref{Mainup2}.
\end{proof}


\section{Green's function estimate}\label{sec5}
In this section, we will apply Schr\"odinger heat kernel estimates to obtain
Green's function estimates for the Schr\"odinger operator on complete
gradient shrinking Ricci solitons; see Theorem \ref{Green}. Recall that
Malgrange \cite{[Ma]} proved that any Riemannian manifold admits a Green's
function
\[
G(x,y):=\int^\infty_0H(x,y,t)dt
\]
if the integral on the right-hand side converges, where $H(x,y,t)$ denotes the heat kernel
of the Laplace operator. Varopoulos \cite{[Varo]} showed that any complete manifold $(M,g)$
has a positive Green's function only if
\begin{equation} \label{Greennecessary}
\int_1^{\infty}\frac{t}{V_p(t)}dt<\infty
\end{equation}
for some point $p\in M$, where $V_p(t)$ denotes the volume of the geodesic ball $B_t(p)$
with radius $t$ and center at $p$.
When the Ricci curvature of manifolds is nonnegative, Varopoulos \cite{[Varo]} and Li-Yau
\cite{[LY]} proved \eqref{Greennecessary}
is a sufficient and necessary condition for the existence of positive Green's function.

On an $n$-dimensional complete gradient shrinking Ricci soliton $(M,g,f)$, letting
$H^{\mathrm{R}}(x,y,t)$ be the Schr\"odinger heat kernel of the operator
$L=-\Delta+a\mathrm{R}$ with $a\geq \frac 14$, the Green's function of $L$ is defined by
\[
G^{\mathrm{R}}(x,y)=\int^\infty_0H^{\mathrm{R}}(x,y,t)dt
\]
if the integral on the right-hand side converges. By the Schr\"odinger heat kernel
 estimates, it is easy to get an upper estimate for
the Green's function of $L$, which is similar to Li-Yau estimate \cite{[LY]} of the
classical Green's function.

\begin{proof}[Proof of Theorem \ref{Green}]
Using the definition of $G^{\mathrm{R}}(x,y)$ and Theorem \ref{Mainup2}, we have
\begin{equation*}
\begin{aligned}
G^{\mathrm{R}}(x,y)&=\int^\infty_0H^{\mathrm{R}}(x,y,t)dt\\
&=\int^{r^2}_0H^{\mathrm{R}}(x,y,t)dt+\int^\infty_{r^2}H^{\mathrm{R}}(x,y,t)dt\\
&\le\int^{r^2}_0H^{\mathrm{R}}(x,y,t)dt+\frac{Ae^{-\mu}}{(4\pi)^{n/2}}\int^\infty_{r^2}t^{-\frac n2}dt\\
&\le\frac{Ae^{-\mu}}{(4\pi)^{n/2}}\left[\int^{r^2}_0t^{-\frac
n2}\exp\left(\frac{-r^2}{ct}\right)dt+\int^\infty_{r^2}t^{-\frac n2}dt\right],
\end{aligned}
\end{equation*}
where $r=d(x,y)$ is the distance function from $x$ to $y$, $c>4$ is a constant,
and $A=A(n,c)$ is determined by \eqref{low2}. Notice that for the first term of the right-hand side
of the above inequality, letting $s=r^4/t$ and observing that $r^2<s<\infty$,
we get
\begin{equation*}
\begin{aligned}
\int^{r^2}_0t^{-\frac n2}\exp\left(\frac{-r^2}{ct}\right)dt
&=\int^\infty_{r^2}\left(\frac{r^4}{s}\right)^{-\frac n2}
\exp\left(\frac{-s}{cr^2}\right)\frac{r^4}{s^2}ds\\
&=\int^\infty_{r^2}s^{-\frac n2}\left(\frac{s}{r^2}\right)^{n-2}
\exp\left(\frac{-s}{cr^2}\right)ds\\
&\le c(n)\int^\infty_{r^2}s^{-\frac n2}ds,
\end{aligned}
\end{equation*}
where in the last line we have used the fact that the function $l^{n-2}e^{-l/c}$,
$l\in [1,\infty)$, is bounded from above. Therefore, for $n\ge 3$
\begin{equation*}
\begin{aligned}
G^{\mathrm{R}}(x,y)
&\le\frac{C(n)Ae^{-\mu}}{(4\pi)^{n/2}}\int^\infty_{r^2}t^{-\frac n2}dt\\
&=\frac{2C(n)Ae^{-\mu}}{(n-2)(4\pi)^{n/2}r^{n-2}}.
\end{aligned}
\end{equation*}
and the result follows.
\end{proof}
\section{eigenvalue estimate}\label{sec6}
In this section, we will apply Gaussian upper bounds on the Schr\"odinger heat kernel
$H^{\mathrm{R}}(x,y,t)$ on compact gradient shrinking Ricci solitons to get the
eigenvalue estimates for the Schr\"odinger operator $L$; see Theorem \ref{eigen}.
The proof is essentially parallel to the Li-Yau's Laplace situation on manifolds
\cite{[LY]}.

\begin{proof}[Proof of Theorem \ref{eigen}]
By Theorem \ref{Mainup}, the Schr\"odinger heat kernel of the operator $L$ has an upper bound
\begin{equation}\label{upperh}
H^{\mathrm{R}}(x,y,t)\le\frac{e^{-\mu}}{(4\pi t)^{\frac n2}}.
\end{equation}
Notice that the Schr\"odinger heat kernel can be written as
\[
H^{\mathrm{R}}(x,y,t)=\sum^\infty_{i=1}e^{-\lambda_it}\varphi_i(x)\varphi_i(y),
\]
where $\varphi_i$ is the eigenfunction corresponding to the eigenvalue $\lambda_i$
with $\|\varphi_i\|_{L^2}=1$. Integrating both sides of \eqref{upperh}, we have
\begin{equation*}
\begin{aligned}
\sum^\infty_{i=1}e^{-\lambda_it}\leq \frac{e^{-\mu}}{(4\pi t)^{\frac n2}}V(M),
\end{aligned}
\end{equation*}
where $V(M)$ is the volume of the manifold $M$. Hence,
\[
ke^{-\lambda_kt}\leq \frac{e^{-\mu}}{(4\pi t)^{\frac n2}}V(M),
\]
which further implies
\begin{equation}\label{heaest}
\frac{ke^{\mu}}{V(M)}\le e^{\lambda_kt}(4\pi t)^{-\frac n2}
\end{equation}
for any $t>0$. It is easy to see that the function $e^{\lambda_kt}(4\pi t)^{-\frac n2}$ takes
its minimum at
\[
t_0=\frac{n}{2\lambda_k}.
\]
Plugging this point into \eqref{heaest} gives the lower bound for $\lambda_k$.
\end{proof}

\vspace{.1in}

\textbf{Acknowledgements}.
The author thanks Professor Qi S. Zhang for helpful suggestions.
The author also thanks the referee for making valuable comments and suggestions and pointing
out many errors which helped to improve the presentation of this work. This work was
partially supported by NSFS (17ZR1412800) and NSFC (11671141).

\vspace{2em}

\textbf{Declarations}

\

\textbf{Conflict of interest} The author declares that there is no conflict of
interest.

\vspace{.1in}

\textbf{Code availability} Not applicable.

\vspace{.1in}

\textbf{Data availability statement} Data sharing not applicable to this article as
no datasets were generated or analysed during the current study.

\bibliographystyle{amsplain}

\end{document}